\documentclass[12pt]{article}
\usepackage[margin=1in]{geometry} 
\usepackage{amsmath,amsthm,amssymb}
\usepackage[english]{babel}
\usepackage[utf8]{inputenc} 
\usepackage{algorithm}
\usepackage{algorithmic}
\usepackage{graphicx}
\usepackage{hhline}
\graphicspath{ {./images/} }
\usepackage{tabularx}
\usepackage{amsmath,arydshln}
\usepackage{color,hyperref}
\usepackage{ragged2e,booktabs,caption}
\usepackage{array}
\usepackage{cleveref}

\newcommand{\sF}{{\mathcal F}}

\newcommand{\F}{{\mathbb{F}}}

\newcommand{\Z}{{\mathbb Z}}

\usepackage{setspace}
\theoremstyle{definition}
\newtheorem{theorem}{Theorem}
\newtheorem{definition}{Definition}
\newtheorem{corollary}{Corollary}
\newtheorem{lemma}{Lemma}

\newtheorem{example}{Example}

\theoremstyle{remark} \newtheorem*{remark}{Remark}

\providecommand{\keywords}[1]
{
	\small	
	\textbf{{Keywords:}} #1
}
\providecommand{\amssub}[1]
{
	\small	
	\textbf{{Subject Classification:}} #1
}
\catcode`\_=11\relax

\def\_email#1@#2\q_nil{%
	\href{mailto:#1@#2}{{\emailfont #1\emailampersat #2}}
}
\newcommand\emailfont{\sffamily}
\newcommand\emailampersat{{\color{black}\small@}}
\catcode`\_=8\relax  
\usepackage{times}
\usepackage{amsthm}

\usepackage{amssymb,amsmath}

\newcommand{\be}{\begin{eqnarray}}
\newcommand{\ee}{\end{eqnarray}}
\newcommand{\nn}{{\nonumber}}

\DeclareMathOperator{\Tr}{Tr}
\DeclareMathOperator{\rTr}{rTr}

\begin{document}
\title{The number of irreducible polynomials over finite fields with vanishing trace and reciprocal trace}
\author{Yağmur Çakıroğlu and Oğuz Yayla \\
	\small \texttt{Dept.\ of Mathematics, Hacettepe University},\\
	 \small \texttt{06800, Beytepe, Ankara, Turkey}\\
	\small \texttt{yagmur.cakiroglu@hacettepe.edu.tr}\\
	 \small \texttt{oguz.yayla@hacettepe.edu.tr}
\and Emrah Sercan Yılmaz\\
\small \texttt{Dept.\ of Mathematics, Boğaziçi University},\\
	\small \texttt{34342, Bebek, İstanbul, Turkey}\\
	\small\texttt{emrahsercanyilmaz@gmail.com}}

\maketitle
\begin{abstract} \noindent We present the formula for the number of monic irreducible polynomials of degree $n$ over the finite field $\F_q$ where the coefficients of $x^{n-1}$ and $x$ vanish for $n\ge3$. In particular, we give a relation between rational points of algebraic curves over finite fields and the number of elements $a\in\F_{q^n}$ for which Trace$(a)=0$ and Trace$(a^{-1})=0$. Besides, we apply the formula to give an upper bound on the number of distinct constructions of a family of sequences with good family complexity and cross-correlation measure.\\
\\
\amssub{11T06,11G20,94A55}
\end{abstract}

\section{Introduction}
Let $r$ be a positive integer, $p$ be a prime number and $q=p^r$, $\F_q$ be the finite field with $q$ elements and let $I_q(n)$ denote the number of monic irreducible polynomials of degree $n$ over $\F_q[x]$. 
It is a well-known formula given by Gauss \cite{Gauss} that
\begin{equation}\nn I_q(n)=\frac{1}{n}\sum_{d|n}\mu(d)q^{n/d}.
\end{equation} 
Let $I_q(n,\gamma_1,\ldots,\gamma_k)$ denote the number of monic irreducible polynomials over $\F_q$ of degree $n$ whose first $k$ coefficients following the leading one is prescribed to $\gamma_1, \ldots, \gamma_k\in\F_q$, respectively. Carlitz in \cite{C1952} showed that
\begin{equation} \nn \label{carlitz} 
I_q(n,\gamma)=\frac{1}{qn}\sum_{d|n,p\nmid d}\mu(d)q^{n/d}.
\end{equation}
Kuz'min \cite{K1991,K1994} considered the case of two prescribed coefficients and gave the formula for $I_q(n,\gamma_1,\gamma_2)$. Yucas and Mullen determined the formula for $I_2(n,\gamma_1,\gamma_2,\gamma_3)$ when $n$ is even \cite{YM2004}, later Yucas and Fitzgerald determined the formula for $I_2(n,\gamma_1,\gamma_2,\gamma_3)$ when $n$ is odd \cite{YF2003}. Also, Yucas \cite{Y2006} gave an alternative proof of Carlit'z formula. Ahmadi et al. in \cite{AGGMY2016} gave the formula for $I_{2^r}(n,0,0)$ for all $r\ge 1$. Most recently, Granger present direct and indirect methods for solving the prescribed traces problem for $q=2$ and $n$ odd. And then in \cite{G2019} he applied these methods for $I_q(n,\gamma_1,\gamma_2,\dots,\gamma_l)$ and $l\ge 7$. Also he obtained explicit formulas for $l=3$ where $q=3$. 
Let $\bar{I}_q(n,\gamma_1,\gamma_2)$ denote the number of monic irreducible polynomials over $\F_q$ of degree $n$ with the coefficients of $x^{n-1}$ and $x$ being the prescribed values $\gamma_1,\gamma_2$, respectively. In this paper we give the formula for $\bar{I}_q(n,0,0)$. Besides, we use this formula to present an upper bound on the number of distinct families with good pseudorandom measures such as family complexity and cross-correlation.



The paper is organized as follows. We present some definitions and previous results in Section~\ref{sec:Pre}. 
In Section \ref{sec:lpol} we present  the concept of $L$-polynomial  of algebraic curves over $F_q$ and its connection to the number of rational points on the algebraic curve. 
In Section \ref{sec:main}, we present our main result and prove the formula on the number of irreducible polynomials with vanishing trace and reciprocal trace. 
In Section \ref{sec:examples} we give  examples and tables for $q=4$ and $q=9$. 
In Section \ref{sec:pseudo} we give a  result on the number of distinct families of pseudorandom sequences with good family complexity and cross-correlation measure.


\section{Preliminaries}\label{sec:Pre}
For $a \in \F_{q^n}$, let the characteristic polynomial of $a$ over $\F_q$ be
\begin{equation}\nn \label{charpol}
\prod_{i=0}^{n-1}(x-a^{q^i})=x^n-a_{n-1}x^{n-1}+\cdots+(-1)^{n-1}a_1x+(-1)^na_0.
\end{equation}
Then we define trace and reciprocal-trace of $a\in \F_{q^n}$ to the base field $\F_q$ as $\Tr(a):=a_{n-1}$ and $\rTr(a):=a_{1}/a_0$, respectively. Hence, we have
\begin{equation} \nn
\Tr(a)=\sum_{i=0}^{n-1}a^{q^i} \mbox{ and }\rTr(a)=\sum_{i=0}^{n-1}a^{-q^i}.
\end{equation}

Let $f(x) =x^n-c_{n-1}x^{n-1}+\cdots+(-1)^{n-1}c_1x+(-1)^n c_0 \in \F_q[x]$ be an irreducible polynomial over $\F_q$. Similarly, we define trace and reciprocal-trace of $f\in \F_{q^n}[x]$ as $\Tr(f):=c_{n-1}$ and $\rTr(f):=c_{1}/c_0$, respectively.

For $\gamma_1,\gamma_2\in\F_q$, let $F_q(n,\gamma_1,\gamma_2)$ be the number of elements $a\in\F_q^n$ for which $\Tr(a)=\gamma_1$ and $\rTr(a)=\gamma_2.$
In this paper we will first consider the values of $F_q(n,\gamma_1,\gamma_2)$ and give its formula for $\gamma_1=0$ and $\gamma_2=0$. Before that, we give some definitions and preliminary results. 
We begin with the definition of Möbius function.
\begin{definition}\cite[Definition 2.1.22]{HBFF2013} The Möbius $\mu$ function is defined on the set of positive integers by
		\begin{displaymath}
	    \mu(m)= \left\{ \begin{array}{ll}
		1 & \textrm{if $m=1$}\\
		(-1)^{k} & \textrm{if $m=m_1m_2\dots m_k$ where the $m_i$ are distinct primes}\\
		0 & \textrm{if $p^2$ divides $m$ for some prime $p$}
		\end{array}\right.
		\end{displaymath} 
\end{definition}
\begin{lemma}\cite[Theorem 2.25]{LN1996}\label{LidlNied} Let $F$ be a finite extension of $K=\F_q$. Then for $a\in F$ we have $\Tr(a)=0$ if and only if $a=y^q-y$ for some $y\in F $.
\end{lemma}


We note that for a positive integer $n$ with a positive divisor $d$ and $P$ be a polynomial of degree $n/d$, the following trivially holds \be \label{eq:trace} \Tr(P^d)=d\cdot \Tr(P) \text { and } \rTr(P^d)=d\cdot \rTr(P).\ee

We now present an analog result of \cite[Theorem 1]{AGGMY2016} in the following theorem. Since the proof is not direct, we give it here.
\begin{theorem}\label{numirr}
	Let $n\ge 2$ be an integer. Then 
	\[
	\bar{I_q}(n,0,0)= \frac1n\sum_{d\mid n, p\nmid d} \mu(d)\left(F_q(n/d,0,0)-[p \text{ divides }n]q^{n/pd}\right).
	\]
\end{theorem}
\begin{proof}
	We have 	\begin{align*}
	\hspace{-0.8cm}F_q(n,0,0)&=\left|\bigcup_{\beta \in \mathbb F_{q^n}, \: \Tr(\beta)=0, \rTr(\beta)=0} \text{Min}(\beta)\right|\\
	&=\left|\bigcup_{d\mid n}\frac nd \left\{P\in \text{Irr}\left(\frac nd\right) : d\cdot \Tr(P)=0,\: d\cdot \rTr(P)=0 \right\}\right|\\
	&=[p \text{ divides } n]\left|\bigcup_{d\mid n, \; p \mid d}\frac nd \left\{P\in \text{Irr}\left(\frac nd\right) \right\}\right| \\
	&\qquad +\left|\bigcup_{d\mid n, \: p \nmid d}\frac nd \left\{P\in \text{Irr}\left(\frac nd\right) : \Tr(P)=0,\: \rTr(P)=0 \right\}\right|\\
	&=[p \text{ divides } n]\sum\limits_{d\mid n, \: p \mid d}\frac nd \bar{I_q}\left(\frac nd\right)
	+\sum\limits_{d\mid n, \: p\nmid d}\frac nd \bar{I_q}\left(\frac nd,0,0\right)\\&=[p \text{ divides } n]q^{n/p}+\sum\limits_{d\mid n, \: p\nmid d}\frac nd \bar{I_q}\left(\frac nd,0,0\right),
	\end{align*}	where the third equality follows from \eqref{eq:trace}. Therefore, $$\bar{I_q}(n,0,0)=\frac1n\sum\limits_{d\mid n, \: p\nmid d}\left(F_q(n/d,0,0)-[p \text{ divides } n] q^{n/pd}\right).$$
\end{proof}
\section{L-Polynomial}\label{sec:lpol}
In this chapter we define the $L$-Polynomial of curves over a finite field. Also we give a well-known formula for the number of rational points on algebraic curves over the finite fields. 
\begin{definition} Let $q=p^r$ where $p$ is a prime number. Let $C=C(\F_q)$ be a (projective, smooth, absolutely irreducible) algebraic curve of genus g defined over $\F_q$. Consider the $L$-polynomial of the curve $C$ over $\F_q$ defined by 
$$L_{C}(t)=\exp{\bigg(\sum_{n=1}^\infty(\#C(\F_{q^n})-q^n-1)\frac{t^n}{n}\bigg)}$$
where $\#C(\F_{q^n})$ denotes the number of $\F_{q^n}$-rational points of $C$.       
\end{definition} Also $L_{C}(t)$ is defined as follows:
$$L_{C}(t)=\sum_{i=0}^{2g}c_it^i$$ where $c_i\in\Z$ and $g$ is the genus of $C$. For instance, for genus $1$, the $L$-polynomial given by $L_{C}(t)=qt^2+c_1t+1$, where $c_1=\#C(F_q)-(q+1)$. In general, the coefficients of the $L$-polynomial are determined by $\#C(\F_{q^n})$ for $n=1,2,\dots,g$. Let $\alpha_1,\dots,\alpha_{2g}$ be the roots of the reciprocal of the $L$-polynomial of $C$ over $\F_q$. Then
$$L_{C}(t)=\prod_{i=1}^{2g}(1-\alpha_{i}^nt).$$ We also have that 
\begin{equation}\label{ratpoints}\#C(\F_{q^n})=(q^n+1)-\sum_{i=1}^{2g}(\alpha_{i})^n\end{equation} for all $n\ge1$, where $|\alpha_i|=\sqrt{q}$. 
\section{Finding the values $F_q(n,0,0)$}
\label{sec:main}
In this section we will find the numbers $F_q(n,0,0)$ where $q$ is an even prime power and $n$ is a positive integer. We relate these numbers with $q-1$ elliptic curves which are related with trace. Since calculating the number of $\mathbb F_q$-rational points of an elliptic curve is enough to find all the number of $\mathbb F_{q^n}$-rational points, the given formula for $F_q(n,0,0)$ is fast to compute. Since these curves are related with trace, we can prefer to write an algorithm using the trace forms.

We note that exact method can be applied for $F_q(n,t_1,t_2)$ where $q$ is an any prime power and $t_1,t_2 \in \mathbb F_q$.

Let $q$ be an even prime power and $n$ be a positive integer. For functions $q_1,\ q_2 : \mathbb F_{q^n} \to \mathbb F_q$ define related $N(t_1,t_2)$ be the number of elements in $\mathbb F_{q^n}$ satisfying $q_1(x)=t_1$ and $q_2(x)=t_2$. For a function   $f : \mathbb F_{q^n} \to \mathbb F_q$ define $Z(f)$  be the number of elements in $\mathbb F_{q^n}$ satisfying $f(x)=0$.

\begin{lemma}\cite[Lemma 6]{AGGMY2016}\label{gen-q12}
Let $q_1,\ q_2 : \mathbb F_{q^n} \to \mathbb F_q$  be any functions. Then 
\[
N(0,0)=\frac1q\left(Z(q_1)+\sum_{\alpha\in \mathbb F_q}Z(\alpha q_1-q_2)-q^n\right).
\] 
\end{lemma}
\begin{proof}It follows by the following equalities.
	\begin{align*}
	q^n=\sum_{\alpha,\beta \in \mathbb F_q}N(\alpha,\beta)&=\sum_{\beta\in \mathbb F_q}N(0,\beta)+\sum_{\beta\in \mathbb F_q}\sum_{\alpha\in \mathbb F_q^\times}N(\alpha,\beta)\\
	&=Z(q_1)+\sum_{\beta\in \mathbb F_q}\sum_{\alpha\in \mathbb F_q^\times}N(\alpha,\alpha\beta)\\
	&=Z(q_1)+\sum_{\alpha,\beta\in \mathbb F_q}N(\beta,\alpha\beta)-qN(0,0)\\
	&=Z(q_1)+\sum_{\alpha\in\mathbb F_q}Z(\alpha q_1-q_2)-qN(0,0).
	\end{align*}
\end{proof}

\begin{lemma}\label{CZ}
	Let $q_1(x)=\Tr(x)$ and $q_2(x)=\rTr(x)$ be functions from $\mathbb F_{q^n}$ to $\mathbb F_{q}$.  The number of $\mathbb F_{q^n}$-rational points of $x(y^q-y)=\alpha x^2-1$ equals to $qZ(\alpha q_1-q_2)-q+2$.
\end{lemma}
\begin{proof}
	The projective curve $xy^q-xyz^{q-1}=\alpha x^2z^{q-1}-z^{q+1}$ has two infinity points $(1:0:0)$ and $(0:1:0)$ and has no extra solution when $x=0$. If $x\ne 0$, then the points on $x(y^q-y)=\alpha x^2-1$ are related with the the set of zeros of $\Tr(\alpha x-x^{-1})$. If $x$ is a such zero, then there exists $y\in \mathbb F_{q^n}$ such that all the points $(x,y+c)$ are on the curve where $c\in \mathbb F_{q^n}$. Therefore,  $\mathbb F_{q^n}$-rational points of $x(y^q-y)=\alpha x^2-1$ equals to \[2+q(Z(\alpha q_1-q_2)-1)=qZ(\alpha q_1-q_2)-q+2.\]
\end{proof}
\begin{lemma}\label{CaC}
	Assume that $q$ is an even prime power. Let $\alpha \in \mathbb F_{q}^\times$. The number of $\mathbb F_{q^n}$-rational points of the curves $x(y^q+y)=\alpha x^2+1$ and $x(y^q+y)= x^2+1$ over $\mathbb F_q$ are same.
\end{lemma}
\begin{proof}
Since order of $\alpha$ is odd, there exist $n$ such that $2n+1$ is the order of $\alpha$. The transformation  $(x,y)\to (\alpha^n x,\alpha^{-n}y)$ on $x(y^q+y)=\alpha x^2+1$ gives $x(y^q+y)= x^2+1$.
\end{proof}
The following lemma follows by Lemma $8$ in \cite{AGGMY2016}. 
\begin{lemma}\label{CCa} 
	Assume that $q$ is an even prime power.  The curve $C:x(y^q+y)= x^2+1$ over $\mathbb F_q$ is the fiber product of the curves $C_{\alpha}:x(y^2+y)= \alpha(x^2+1)$ over $\mathbb F_q$  where $\alpha \in \mathbb F_q^\times$. Therefore, 
	\[
	\#C(\mathbb F_{q^n})-(q^n+1)=\sum_{\alpha\in \mathbb F_q^\times}\left(\#C_\alpha(\mathbb F_{q^n})-(q^n+1)\right). 
	\]
\end{lemma}
The following lemma follows by an analogue of Lemma $8$ in \cite{AGGMY2016} to all primes. 
\begin{lemma}\label{CCap}
Assume that $q$ is prime $p$-power. Let $\alpha \in \mathbb F_{q}^\times$. The curve $C_\alpha:x(y^q-y)= \alpha x^2-1$ over $\mathbb F_q$ is the fiber product of the curves $C_{\alpha,\beta}:x(y^p-y)= \beta(\alpha x^2-1)$ over $\mathbb F_q$  where $\beta\in \mathbb F_q^\times/\mathbb F_p^\times$ as a representative set  in $F_q^\times$. Therefore, 
	\[
	\#C_\alpha(\mathbb F_{q^n})-(q^n+1)=\sum_{\beta\in \mathbb F_q^\times/\mathbb F_p^\times}\left(\#C_{\alpha,\beta}(\mathbb F_{q^n})-(q^n+1)\right). 
	\]
\end{lemma}

\begin{theorem}\label{binary}

Assume that $q$ is an even prime power.  Let $C_{\alpha}:x(y^2+y)= \alpha(x^2+1)$ be curves over $\mathbb F_q$ for $\alpha \in \mathbb F_q^\times$. Define $S_\alpha(\mathbb F_{q^n})=\#C_\alpha(\mathbb F_{q^n})-(q^n+1)$.  Then \[
F_q(n,0,0)=q^{n-2}+\frac{q-1}{q^2}\sum_{\alpha\in \mathbb F_q^\times}\left(S_\alpha(\mathbb F_{q^n})+1\right).
\]	
\end{theorem}
\begin{proof}
	Let $q_1(x)=\Tr(x)$ and $q_2(x)=\rTr(x)$ be functions from $\mathbb F_{q^n}$ to $\mathbb F_{q}$. By Lemma \ref{gen-q12}
	\begin{align*}
	qF_q(n,0,0)&=Z(q_1)+Z(q_2)+\sum_{\alpha\in\mathbb F_q^\times}Z(\alpha q_1+q_2)-q^n\\&=q^{n-1}+q^{n-1}+\sum_{\alpha\in\mathbb F_q^\times}Z(\alpha q_1+q_2)-q^n\\ &=
	q^{n-1}+\sum_{\alpha\in\mathbb F_q^\times}\left(Z(\alpha  q_1+ q_2)-q^{n-1}\right).
	\end{align*} By Lemma \ref{CZ} and Lemma \ref{CaC} 
		\begin{align*}
	qF_q(n,0,0)
	&=	q^{n-1}+\sum_{\alpha\in\mathbb F_q^\times}\left(\frac{\#C(\mathbb F_{q^n})+q-2}{q}-q^{n-1}\right)\\
	&=	q^{n-1}+\frac{q-1}{q}\left(\#C(\mathbb F_{q^n})-(q^n+1)+q-1\right)
	\end{align*}
	By Lemma \ref{CCa}
	\begin{align*}
	F_q(n,0,0)
	&=q^{n-2}+\frac{q-1}{q^2}\left(\left(\sum_{\alpha\in \mathbb F_q^\times}\left(\#C_\alpha(\mathbb F_{q^n})-(q^n+1)\right)\right)+q-1\right)\\
	&=q^{n-2}+\frac{q-1}{q^2}\sum_{\alpha\in \mathbb F_q^\times}\left(S_\alpha(\mathbb F_{q^n})+1\right).
	\end{align*}
	
\end{proof}
Similarly, we can prove the following theorem. We will skip similar calculation details.
\begin{theorem}\label{tekler}

Assume that $q$ is prime $p$-power. Let $C_{\alpha,\beta}:x(y^p-y)= \beta(\alpha x^2-1)$ be curves over $\mathbb F_q$ for $\alpha \in \mathbb F_q^\times$ and $\beta \in F_q^\times/F_p^\times$ as representative set in $F_q^\times$. Define $S_{\alpha,\beta}(\mathbb F_{q^n})=\#C_{\alpha,\beta}(\mathbb F_{q^n})-(q^n+1)$.  Then \[
	F_q(n,0,0)=q^{n-2}+\frac{(q-1)^2}{q^2}+\frac1{q^2}\sum_{\alpha\in\mathbb F_q^\times}\sum_{\beta\in \mathbb F_q^\times/\mathbb F_p^\times}S_{\alpha,\beta}(\mathbb F_{q^n})
	\]	
\end{theorem}
\begin{proof}
	Let $q_1(x)=\Tr(x)$ and $q_2(x)=\rTr(x)$ be functions from $\mathbb F_{q^n}$ to $\mathbb F_{q}$. By Lemma \ref{gen-q12}
	\begin{align*}
	qF_q(n,0,0)&=
	q^{n-1}+\sum_{\alpha\in\mathbb F_q^\times}\left(Z(\alpha  q_1-q_2)-q^{n-1}\right).
	\end{align*} By Lemma \ref{CZ} and Lemma \ref{CCap}
	\begin{align*}
	F_q(n,0,0)
	&=	q^{n-2}+\frac{(q-1)^2}{q^2}+\frac{1}{q^2}\sum_{\alpha\in\mathbb F_q^\times}\left(\#C_\alpha(\mathbb F_{q^n})-(q^n+1)\right)\\
	&=q^{n-2}+\frac{(q-1)^2}{q^2}+\frac1q\sum_{\alpha\in \mathbb F_q^\times}\left(\sum_{\beta\in \mathbb F_q^\times/\mathbb F_p^\times}\left(\#C_{\alpha,\beta}(\mathbb F_{q^n})-(q^n+1)\right)\right)\\
	&=q^{n-2}+\frac{(q-1)^2}{q^2}+\frac{1}{q^2}\sum_{\alpha\in\mathbb F_q^\times}\sum_{\beta\in \mathbb F_q^\times/\mathbb F_p^\times}S_{\alpha,\beta}(\mathbb F_{q^n}).
	\end{align*}	
\end{proof}
\begin{remark}
    We note that the curve $C_{\alpha}:x(y^2+y)= \alpha(x^2+1)$  over $\mathbb F_{2^r}$ for $\alpha \in \mathbb F_q^\times$ is non-singular. Therefore by using genus-degree formula it has genus  1. On the other hand, for an odd prime power $q=p^r$, the curve $C_{\alpha,\beta}:x(y^p-y)= \beta(\alpha x^2-1)$ over $\mathbb F_q$ for $\alpha \in \mathbb F_q^\times$ and $\beta \in F_q^\times/F_p^\times$ has genus $p-1$. This can be seen form \cite[Theorem 3.7.8]{stichtenoth2009algebraic} as the Artin-Schreier
extension of the rational function field $\F_q(x)$ defined by $y^p-y= \beta(\alpha x^2-1)/x$ has only ramified rational places $x$ and $1/x$ of $\F_q(x)$.
\end{remark}
\section{Examples}
In this section, we illustrate Theorems \ref{numirr} and \ref{tekler} for $q=4$ and $q=9$, respectively.
\label{sec:examples}
\begin{example}
Let $q=4$ and $n=5$. Let $C_{\alpha,\beta}:x(y^2-y)= \alpha( x^2+1)$ be curves over $\mathbb F_4$ for $\alpha \in \mathbb F_4^\times$. Let $S_{\alpha}(\mathbb F_{q^n})$ be defined as in Theorem \ref{binary}. Then by using Magma \cite{magma} we get
\begin{equation}\sum_{\alpha\in\mathbb F_4^\times}S_{\alpha}(\mathbb F_{4^5}+1)=-176.   
\end{equation}
Then  Theorem \ref{binary} gives \begin{equation}\label{rslt1}
 F_4(5,0,0)=64-\frac{528}{16}=31.   
\end{equation}
On the other hand, we, in Table \ref{tbl1}, tabulate the number of elements in $\F_{4^n}$ with both vanishing  trace and reciprocal trace. We get the values in Table \ref{tbl1} by exhaustive counting. Note that \eqref{rslt1}  complies with the corresponding value in the Table \ref{tbl1}.
\begin{table}[!h]
\centering
\caption{The values of $F_4(n,0,0)$.}
\begin{tabular}{|c|c|c|c|c|c|c|c|c|c|}
\hline
n & 3 & 4 & 5 & 6 & 7 & 8 & 9 & 10 \\ \hline
$F_4(n,0,0)$ & 7 & 16 & 31 & 268 & 1135 & 4096 & 16279 & 64684 \\ \hline
\end{tabular}
\label{tbl1}
\end{table}

Now we calculate the number of monic irreducible polynomials of degree $5$ in $\F_4[x]$ with  vanishing  trace and reciprocal trace. By  Theorem \ref{numirr} we have 
\begin{equation}
\bar{I_4}(5,0,0)=\frac{1}{5}\big(\mu(5)(F_4(1,0,0)+\mu(1)(F_4(1,0,0)\big).
\end{equation}
Besides, Theorem \ref{binary} gives $F_4(1,0,0)=1$ and $F_4(5,0,0)=31$ as above. Therefore we get
\begin{equation}\label{polresult1}
\bar{I_4}(5,0,0)=6.
\end{equation}
Similarly, we counted exhaustively the number $\bar{I_4}(n,0,0)$ of irreducible polynomials for $n=3,4,\ldots,10$ and tabulated them in Table \ref{tbl2}. We see that \eqref{polresult1} complies with the value given in Table \ref{tbl2}. 
\begin{table}[!h]
\centering
\caption{The number $\bar{I_4}(n,0,0)$ of monic irreducible polynomials}
\begin{tabular}{|c|c|c|c|c|c|c|c|c|c|}
\hline
n & 3 & 4 & 5 & 6 & 7 & 8 & 9 & 10 \\ \hline
$\bar{I_4}(n,0,0)$ & 0 & 0 & 6 & 34 & 162 & 480 & 1808 & 6366 \\ \hline
\end{tabular}
\label{tbl2}
\end{table}
\end{example}
\begin{example}
Let $q=9$ and $n=5$. Let $C_{\alpha,\beta}:x(y^3-y)= \beta(\alpha x^2-1)$ be curves over $\mathbb F_9$, where $\alpha \in \mathbb F_9^\times$ and the elements $\beta$ are the representatives of the quotient group $ F_9^\times/F_3^\times$. Let $S_{\alpha,\beta}(\mathbb F_{q^n})$ be defined as in Theorem \ref{tekler}. Then by using Magma \cite{magma} we get \begin{equation}\sum_{\alpha\in\mathbb F_9^\times}
 \sum_{\beta\in \mathbb F_9^\times/\mathbb F_3^\times}S_{\alpha,\beta}(\mathbb F_{9^5})=5768.   
\end{equation}
Then by Theorem \ref{tekler} we have \begin{equation}\label{lastresult}
 F_9(5,0,0)=729+\frac{64+5768}{81}=801. 
\end{equation}
We see that \eqref{lastresult} is equal to the value that we obtain by counting the number of elements in $\F_9$ with vanishing trace and reciprocal trace, see Table \ref{tbl3}.
\begin{table}[!h]
\centering
\caption{ The values of $F_9(n,0,0)$.}
\begin{tabular}{|c|c|c|c|c|c|c|c|}
\hline
n & 3 & 4 & 5 & 6 & 7 & 8  \\ \hline
$F_9(n,0,0)$ & 9& 9 & 89& 801 & 6561 &57904 \\ \hline
\end{tabular}
\label{tbl3}
\end{table}

We know by Theorem \ref{numirr} that the number of monic irreducible polynomials of degree $5$ over $\F_9[x]$ with vanishing trace and reciprocal trace satisfies
\begin{equation}
\bar{I_9}(5,0,0)=\frac{1}{5}\big(\mu(5)(F_9(1,0,0)+\mu(1)(F_9(5,0,0)\big).
\end{equation}
By  Theorem \ref{tekler} we get $F_9(1,0,0)=1$ and $F_9(5,0,0)=801$. Therefore we obtain
\begin{equation}\label{polresult2}
\bar{I_9}(5,0,0)=160.
\end{equation}
Then also we see that the values in Table \ref{tbl4} and  \eqref{polresult2} are equal.
\begin{table}[h]
\centering
\caption{The number $\bar{I_9}(n,0,0)$ of monic irreducible polynomials}
\begin{tabular}{|c|c|c|c|c|c|c|c|c|}
\hline
n & 3 & 4 & 5 & 6 & 7 & 8 & 9  \\ \hline
$\bar{I_9}(n,0,0)$ & 0 & 0 & 160 & 1080 & 8272 &66500 &  530592 \\ \hline
\end{tabular}
\label{tbl4}
\end{table}
\end{example}
\section{Pseudorandom sequences}
\label{sec:pseudo}
Pseudorandom sequence is a sequence of numbers generated deterministically and looks random. 
The quality of a pseudorandom sequence are screened not only by statistical test packages 
(for example L'Ecuyer's TESTU01 \cite{l2007testu01}, Marsaglia's Diehard \cite{marsaglia1996diehard} or the NIST battery \cite{nist}) but also by theoretical results on certain measures of pseudorandomness, see \cite{Gya2013,TW2007} and references therein.

In some applications such as cryptography we need a large family of good pseudorandom sequences and we need to provide some bounds on several figures of merit \cite{S2017}. 
In this section we consider the family complexity (short $f$-complexity) and the cross-correlation measure of order $\ell$ of families of sequences.
We start with their definitions and then we define a family of sequences with good $f$-complexity and the cross-correlation measure. In this section we give an upper bound on the number of distinct families by using Theorems \ref{numirr} and \ref{tekler}.
Ahlswede et al.~\cite{AKMS2003} introduced the $f$-complexity as follows.

\begin{definition}
The \textit{$f$-complexity} $C(\sF)$ of a family $\sF$ of binary sequences $E_N \in \{-1,+1\}^N$ of length $N$ 
is the greatest integer $j \geq 0$ such that for any $1 \leq i_1 < i_2< \cdots < i_j \leq N$ and any $\epsilon_1,\epsilon_2, \ldots,  \epsilon_j \in \{-1,+1\}$ 
there is a sequence $E_N = \{e_1,e_2,\ldots , e_N\}\in \sF$ with $$e_{i_1}=\epsilon_1,e_{i_2}=\epsilon_2, \ldots ,e_{i_j}=\epsilon_j.$$
\end{definition}
It is easy to see that $
2^{C(\sF)} \leq |\sF|$,
where $|\sF|$ denotes the size of the family $\sF$.

Gyarmati et al.~\cite{GMS2014} introduced the cross-correlation measure of order $\ell$.


\begin{definition} \label{def.ccm}
The \textit{cross-correlation measure of order $\ell$} of a family $\sF$ of binary sequences $E_{i,N} = (e_{i,1},e_{i,2},\ldots , e_{i,N}) \in \{-1+1\}^N$,  $i=1,2, \ldots , F$, 
is defined as 
$$
\Phi_\ell(\sF) = \max_{M,D,I}\left| \sum_{n=1}^{M}{e_{i_1,n+d_1} \cdots e_{i_\ell,n+d_\ell}}\right|,
$$
where $D$ denotes an $\ell$ tuple $(d_1,d_2,\ldots , d_\ell)$ of integers such that $0 \leq d_1 \leq d_2 \leq \cdots \leq d_\ell < M+d_\ell \leq N$ and $d_i \neq d_j$ 
if $E_{i,N} = E_{j,N}$ for $i \neq j$ and $I$ denotes an $\ell$ tuple $(i_1,i_2, \ldots , i_\ell)\in\{1,2,\ldots ,F\}^\ell$.
\end{definition}

In \cite{yayla2020families}, a family of sequences of Legendre symbols generated from some irreducible polynomials with high family complexity and small cross-correlation
measure up to a large order $\ell$ was given. Similarly, it was shown that its dual family has good measures. Let $p > 2$ be a prime number, $n \geq 5$ and $\Omega_{p,n}$ be a set of irreducible polynomials over $\F_p$ of degree $n$ defined  as 
$$\Omega_{p,n} =\{ f(x) = x^n + a_2x^{n-2} + a_3x^{n-3} + \cdots + a_{n-2}x^{2} + a_n \in \F_p[x], a_2,a_3 \neq 0\}.$$  Let $f\in \Omega_{p,n}$,
$f_i(X) = i^n f(X/i)$
for  $i \in \{1,2,\ldots , p-1 \}$ and $\sF_f$ be a family of binary sequences defined as 
\be \label{eq:legendre} {\sF_f}=\left\{ \left(\frac{f_i(j)}{p}\right)_{j=1}^{p-1} : i=1,\ldots,p-1\right\},
\ee
and $\overline{\sF_f}$ be the dual of $\sF_f$.
Then it is shown in \cite{yayla2020families} that they have good cross-correlation measure and family complexity as
 $$\Phi_k({\sF_f})\ll n k p^{1/2}\log p  \mbox{ and }  \Phi_k(\overline{\sF_f})\ll n k p^{1/2}\log p$$
for each integer $k \in \{1,2,\ldots , p-1\}$ and
$$C({\sF_f})\ge \left(\frac{1}{2}-o(1)\right) \frac{\log(p/n^2)}{\log 2}  \mbox{ and }  C(\overline{\sF_f})\ge \left(\frac{1}{2}-o(1)\right) \frac{\log(p/n^2)}{\log 2}.$$

We have the family size $|\sF_f| = p$ for the family given in \eqref{eq:legendre}. On the other hand, the number $\#\{\sF_f \vert f \in \Omega_{p,n}\}$ of distinct families that can be constructed as in \eqref{eq:legendre} not known. Here, we give a partial solution for this problem, that is, an upper bound on the number of distinct families. 
\begin{corollary}
Let $C_{\alpha}:x(y^p+y)= \alpha(x^2+1)$ be curves over $\mathbb F_p$ for $\alpha \in \mathbb F_p^\times$. Define $S_\alpha(\mathbb F_{p^n})=\#C_\alpha(\mathbb F_{p^n})-(p^n+1)$.  Then \[\#\{\sF_f \vert f \in \Omega_{p,n}\} < \frac1n\sum_{d\mid n, p\nmid d} \mu(d)\left(F_p(n/d,0,0)-[p \text{ divides }n]p^{n/pd}\right)
,\]
where	
\[ F_p(n,0,0) = p^{n-2}+\frac{(p-1)^2}{p^2}+\frac1{p^2}\sum_{\alpha\in\mathbb F_p^\times}S_{\alpha}(\mathbb F_{p^n}).
	\]	
	
\end{corollary}
\begin{proof}
    The family $\sF$ is constructed by using irreducible polynomials $f \in \overline{I}_p(n,0,0)$. Hence we have the case $q=p^r$ for $r=1$. By Theorem \ref{numirr},  we get the  number of irreducible polynomials in terms of $F_p(n,0,0)$. On the other hand, as $q=p$, Theorem \ref{tekler} gives the result.
\end{proof}

\section{Conclusion}
In this paper, we proved the formula for number $\bar{I}_q(n,0,0)$ of irreducible polynomial of degree $n$ over the finite field $\F_q$, $q=p^r$, such that the terms $x^{n-1}$ and $x$ vanish. Our formula reduces the problem of finding $\bar{I}_q(n,0,0)$ into getting the roots of the L-polynomial of the corresponding algebraic curve defined over $\F_q$. The latter is an easier problem as the genus of the curve is $p-1$. In particular, they are elliptic curves when $q=2^r$ and the L-polynomial has only two roots.
\section*{Acknowledgments}
Yağmur Çakıroğlu and Oğuz Yayla are supported by the Scientific and Technological Research Council of Turkey (TÜBİTAK) under Project No: \mbox{116R026}. Emrah Sercan Yılmaz is also supported by TÜBİTAK under Project No: \mbox{117F274}.

	\bibliographystyle{plain}      
	\bibliography{makale}   
\end{document}